\documentclass[12pt, letterpaper]{article}

\usepackage[margin=1in]{geometry}

\usepackage{setspace}
\usepackage[utf8]{inputenc}
\usepackage[T1]{fontenc}
\usepackage{lmodern} 
\usepackage{microtype} 

\usepackage{amsmath}
\usepackage{amssymb}
\usepackage{amsthm}
\usepackage{mathtools}
\usepackage{bbm} 

\usepackage{graphicx}
\usepackage{booktabs} 
\usepackage{caption}
\usepackage{subcaption}
\usepackage{float}

\usepackage{enumitem}
\setlist[itemize]{noitemsep, topsep=0pt}

\usepackage{hyperref}
\hypersetup{
    colorlinks=true,
    linkcolor=blue,      
    citecolor=blue,      
    urlcolor=blue,       
    pdfborder={0 0 0},   
    pdftitle={Optimal Allocation of Scarce Human Override Capacity},
    pdfauthor={Author Name}
}

\usepackage[authoryear]{natbib}
\usepackage{titlesec}
\titleformat{\section}{\Large\bfseries}{\thesection}{1em}{}
\titleformat{\subsection}{\large\bfseries}{\thesubsection}{1em}{}

\newtheorem{theorem}{Theorem}

\newtheorem{proposition}{Proposition}

\newtheorem{assumption}{Assumption}

\theoremstyle{remark}
\newtheorem{remark}{Remark}

\newcommand{\E}{\mathbb{E}}
\newcommand{\Prob}{\mathbb{P}}
\newcommand{\R}{\mathbb{R}}
\newcommand{\N}{\mathbb{N}}
\newcommand{\Ind}{\mathbbm{1}}
\newcommand{\calS}{\mathcal{S}}
\newcommand{\ThetaSet}{\Theta}
\newcommand{\calQ}{\mathbb{N}_0}

\newcommand{\OpT}{\mathcal{T}}
\newcommand{\OpM}{\mathcal{M}}

\title{\LARGE \bf Operating Imperfect AI: Reliability Drift and Human Congestion}

\author{
    \textbf{Ziyao Wang, Svetlozar T Rachev} \\
    Department of Mathematics and Statistics \\
    Texas Tech University \\
    \texttt{ziywang@ttu.edu}
}

\date{Current Draft: \today}

\begin{document}

\maketitle

\begin{abstract}

The deployment of machine learning in high-stakes services relies on ``human-in-the-loop'' architectures to mitigate algorithmic uncertainty. However, existing static policies fail to address a fundamental tension: algorithms suffer from stochastic ``reliability drift,'' while human override capacity is scarce and congestible. We formulate the management of such systems as a dynamic queueing control problem. The system state is defined by the tuple (queue backlog, reliability regime), and the control variable is a state-dependent risk threshold. We prove that the optimal escalation policy is driven by the endogenous ``Shadow Price of Capacity.'' We establish two key structural monotonicity results: (i) \textit{Congestion Shedding}, where the threshold rises with backlog to sacrifice marginal accuracy for responsiveness; and (ii) \textit{Safety Buffering}, where the threshold lowers during drift to use the queue as a ``risk capacitor.'' Furthermore, we identify a critical ``Capacity Phase Transition'' in the arrival-drift parameter space, beyond which no policy can maintain safety standards without causing structural system failure (infinite queues). Our results provide rigorous operational rules for managing the interface between imperfect algorithms and congested experts.

\vspace{0.5cm}
\end{abstract}


\section{Introduction}

\subsection{The Operational Reality of AI Deployment}

Over the past decade, the integration of Artificial Intelligence and Machine Learning into service operations has transitioned from experimental pilots to mission-critical infrastructure.
In industries as diverse as financial services, healthcare, e-commerce, and digital platform governance, predictive algorithms now handle the vast majority of transactional decisions.
For instance, in credit card fraud detection, algorithms process millions of transactions per second, deciding instantaneously whether to approve or decline a payment.
In social media content moderation, classifiers scan billions of posts daily to identify hate speech or misinformation.
In radiology, deep learning models analyze medical images to flag potential anomalies for physician review.

However, contrary to the popular narrative of full automation, the most successful and robust AI deployments are invariably ``hybrid'' or ``human-in-the-loop'' (HITL) systems.
In these architectures, the AI system does not replace human labor entirely.
Instead, it functions as a triage mechanism: it automates the processing of ``easy'' or ``clear-cut'' cases (where the model's confidence is high), while escalating ``hard,'' ``ambiguous,'' or ``edge-case'' tasks to human experts for manual adjudication.
This hybrid structure seeks to leverage the complementary strengths of humans and machines. 
As noted by \citet{Agrawal2018}, the deployment of AI can be conceptualized as a drop in the cost of prediction, which effectively shifts the value toward judgment—a distinctively human capability.

\subsection{The Economic Gap: From Fixed Costs to Shadow Prices}

While the engineering of predictive models has received immense attention, the \textit{operational control} of the workflow remains a critical blind spot.
The prevailing approach in the ML literature, often framed as ``Learning to Defer'' \citep{Madras2018, Raghu2019}, typically assumes that the cost of querying a human expert is a fixed constant $c$.
We argue that this assumption, while valid in crowdsourcing, is fundamentally flawed for enterprise operations.
In settings like fraud investigation units or hospital departments, the expert team is a fixed resource with a finite service rate.
Therefore, the true cost of utilization is not a fixed wage, but the \textit{congestion externality}: every task added to the queue increases the delay for all subsequent tasks.
This creates a dynamic ``Shadow Price of Capacity'' that fluctuates with the system's backlog—a dimension largely ignored in current literature.

\subsection{The Core Tension: Congestion and Reliability Drift}

This paper is motivated by the failure of static policies to cope with the dynamic realities of modern service operations.
We formulate the management of hybrid intelligence systems as a stochastic control problem characterized by the tension between two opposing forces:

\textbf{1. The Congestion Constraint.}
Human experts are a scarce and rigid resource.
Unlike cloud computing instances which can be auto-scaled in seconds, a team of analysts or radiologists has a finite service rate.
Escalating a task incurs a congestion externality.
In many applications, the resulting delay is extremely costly—leading to customer churn in banking or adverse health outcomes in medicine.
Thus, the decision to escalate is not merely a question of statistical accuracy; it is an economic question of capacity rationing.
A policy that is optimal for a single task in isolation (``escalate to minimize error'') may be disastrous for the system as a whole (``escalate until the queue explodes'').

\textbf{2. The Reliability Drift Constraint.}
The second challenge is that the performance of the automated component is not stationary.
Real-world ML systems operate in dynamic environments where the underlying data distribution shifts over time—a phenomenon known as ``concept drift'' or ``reliability drift'' \citep{Gama2014}.
We model this not just as covariate shift, but as a \textit{regime switch} in the error profile.
When a system enters a drift state, the algorithm makes more errors for the same set of inputs.
The natural, safety-conscious response is to rely more heavily on human supervision.
However, this creates a cruel operational paradox: precisely when the system needs human oversight the most (during drift), the demand for human capacity spikes, potentially overwhelming the finite expert pool.

\subsection{Research Questions and Methodology}

Motivated by these challenges, we seek to answer the following research questions:
\begin{itemize}
    \item \textbf{Q1 (Optimal Control):} How should a decision-maker dynamically adjust the escalation threshold in response to real-time fluctuations in queue backlog and system reliability?
    \item \textbf{Q2 (Structural Properties):} Can we characterize the optimal policy using simple, interpretable rules? Specifically, how does the optimal risk threshold vary with congestion levels (queue length) and drift severity?
    \item \textbf{Q3 (System Stability):} What are the fundamental limits of such hybrid systems? Is there a theoretical boundary in the parameter space beyond which no policy can maintain safety?
\end{itemize}

To answer these questions, we formulate a continuous-time stochastic control model (CTMDP).
We employ uniformization techniques \citep{Lippman1975} to analyze the discrete-time equivalent Bellman equation.
By establishing the supermodularity of the value function, we derive threshold-based optimal policies.
Furthermore, we bridge the gap between discounted and average cost criteria using Tauberian theorems \citep{Sennott1999}, providing a rigorous foundation for managing the interface between imperfect algorithms and congested experts.


\section{Literature Review}

Our work is situated at the intersection of three distinct and mature streams of literature: (i) the operational management of human-AI collaboration, (ii) admission control in queueing systems, and (iii) decision-making under model uncertainty. In this section, we review the relevant work in each domain, highlighting the specific gaps our paper addresses.

\subsection{Human-AI Collaboration and Hybrid Intelligence}

A rapidly growing body of literature in Information Systems (IS) and Management Science examines the economic and operational implications of deploying AI alongside humans. The foundational work in this area often focuses on the \textit{complementarity} between human judgment and algorithmic prediction.

\citet{Agrawal2018} conceptualize AI as a drop in the cost of prediction, which effectively shifts the value toward judgment—a distinctively human capability. Following this, several papers have modeled the ``delegation'' or ``triage'' problem. \citet{Madras2018} and \citet{Raghu2019}, in the computer science domain, propose the ``Learning to Defer'' framework, where a classifier is trained not just to predict, but to decide whether to predict or defer to an expert. Their objective function typically minimizes a joint loss: $\min \mathcal{L} = \alpha \cdot \text{Cost}_{\text{machine}} + (1-\alpha) \cdot \text{Cost}_{\text{human}}$.

In the operations literature, \citet{Ibanez2018} empirically study how radiologists behave when assisted by AI. They find that discretionary ordering of tasks can improve throughput but note that humans may over-rely on or under-rely on algorithmic advice. \citet{Fugener2021} explore the limits of human-in-the-loop systems, highlighting that cognitive bottlenecks can negate the benefits of AI if not managed correctly. They show that showing humans the AI's confidence score can sometimes lead to biased decision-making.

\textbf{Critique and Differentiation:} A pervasive limitation in this stream of research is the modeling of human capacity. Most existing models (including \citet{Madras2018} and \citet{Raghu2019}) assume that the cost of querying a human is a fixed constant $c$. This implies either an infinite supply of human effort or a pay-per-task crowdsourcing model. In contrast, most enterprise applications involve a fixed team of experts (e.g., a fraud investigation unit) where capacity is defined by a service rate ($\mu$), not a marginal cost. In such settings, the true cost of utilization is \textit{delay}, which is non-linear and state-dependent. Our work contributes to this stream by explicitly modeling the human component as a queueing system, thereby endogenizing the cost of human intervention. We show that the decision to defer to a human should depend not only on the instance capability (``is the human better at this task?'') but also on the instantaneous system load (``can we afford the wait?'').

\subsection{Admission Control and Triage in Service Systems}

From a methodological perspective, our problem is a variation of the admission control or ``gatekeeper'' problem in queueing networks. This is a classical topic in Operations Research.

\citet{Miller1969} provided one of the earliest characterizations of optimal control in M/M/s systems with rejection options. He showed that for a system with $K$ customer classes, the optimal policy is a threshold policy on the queue length. \citet{Stidham1985} generalized these results to general queueing networks, establishing the optimality of monotonic control policies under submodular cost structures.

More recently, \citet{Alizon2009} studied strategic capacity rationing where a service provider can reject low-value customers to reserve capacity for high-value ones. \citet{Hasija2010} and \citet{Gurvich2009} examined routing and staffing in call centers, a domain that shares structural similarities with our problem (routing calls between automated IVR and human agents). \citet{Argon2009} analyzed priority assignment in emergency triage under imperfect information, focusing on how to sequence jobs when their true urgency is noisy.

\textbf{Critique and Differentiation:} While these papers provide the foundational machinery for threshold policies, they typically assume \textit{stationary} environments. The arrival process and the value/cost distributions are assumed to be drawn from fixed parameters ($\lambda, F_V$). Our model introduces a novel dimension specific to the AI context: \textit{reliability drift}. In our framework, the ``cost of rejection'' (i.e., the cost of automating a task) is not a static parameter but a stochastic process itself. This requires a controller that anticipates regime shifts. The most closely related work in this regard is by \citet{Huh2009}, who study dynamic pricing with inventory constraints, but in a different context. We extend the admission control literature by analyzing how operations must adapt when the underlying ``production technology'' (the algorithm) becomes unreliable.

\subsection{Operational Adaptation to Model Drift}

The third stream of literature concerns the management of model risk and concept drift. In the machine learning community, concept drift is typically treated as a statistical estimation problem, solved via online learning, windowing techniques, or ensemble methods \citep{Gama2014}. The operational response is usually implicitly assumed to be ``retrain the model.''

However, retraining is not instantaneous. There is often a significant lag—ranging from hours to weeks—between the detection of drift and the deployment of a patched model. During this latency period, the firm must continue to operate. Our work addresses this \textit{operational gap}. We argue that before the model can be fixed statistically, it must be managed operationally. This connects our work to the literature on robust operations and disruption management. Just as supply chain managers buffer inventory against supply disruptions \citep{Tomlin2006}, we propose buffering human capacity to protect against ``informational disruptions'' (drift). We formalize the intuition that human queues act as a safety capacitor for AI systems.

\citet{Besbes2021} study a related problem of ``optimal dual-sourcing with model uncertainty,'' where a firm sources from a known supplier or an unknown supplier. While structurally similar (choice between two sources), their focus is on \textit{learning} the unknown parameter. In our model, the drift state is dynamic (Markovian switching), requiring a policy that is not just learning-based but resilient to recurring shocks.

\subsection{Positioning and Summary}

To summarize, our work sits at the unique intersection of these three fields. We take the ``risk score'' concept from the AI literature, the ``threshold policy'' tools from the Queueing literature, and the ``regime switching'' concept from the Reliability literature to build a comprehensive theory of Hybrid System Operations.



\section{Model Formulation} \label{sec:model}

We consider a continuous-time stochastic control framework for a hybrid decision-making system. The system consists of a central decision agent (the ``controller'') who acts as a gatekeeper between an imperfect automated classification algorithm and a finite-capacity pool of human experts. The operational timeline is continuous, denoted by $t \in [0, \infty)$.

\subsection{Task Arrivals and Heterogeneous Risk Profiles}

Tasks arrive to the system according to a Poisson process with rate $\lambda > 0$. Let $\{\tau_n\}_{n \ge 1}$ denote the sequence of arrival times. Upon arrival, each task $i$ is characterized by an observable attribute vector $x_i \in \mathcal{X}$, which the automated system processes to generate a prediction $\hat{y}_i$ and a scalar \textit{risk score} $s_i \in \calS$, where $\calS = [0, 1] \subset \R$.

The risk score $s_i$ serves as a sufficient statistic for the epistemic uncertainty of the automated model regarding the specific task. In practice, this could be the calibrated probability of error (e.g., $1 - \max_k P(y=k|x)$) or a model uncertainty metric derived from Bayesian ensembles. We assume that the risk scores of arriving tasks are independent and identically distributed (i.i.d.) random variables drawn from a cumulative distribution function $F_S(s)$ with continuous density $f_S(s)$ that is strictly positive on $(0, 1)$.

\begin{remark}[Observability of Risk Scores]
We assume the risk score $s_i$ is instantly observable upon arrival. This is consistent with modern MLOps pipelines where inference (generating $s_i$) is computationally instantaneous ($< 100$ms) relative to the time scale of human review (minutes). The decision to escalate is thus made \textit{after} the model's initial pass but \textit{before} the final decision is committed.
\end{remark}

\subsection{Reliability Drift Dynamics}

A central feature of our model is that the mapping from risk scores to actual error probabilities is not constant. The automated system operates in a non-stationary environment characterized by \textit{reliability drift}. We model the system's reliability environment as a continuous-time Markov chain (CTMC) $\{\theta(t), t \ge 0\}$ taking values in a finite state space $\ThetaSet = \{1, 2, \dots, K\}$.

We order the states such that $\theta = 1$ corresponds to the ``nominal'' or ``stable'' regime, where the automated system performs as calibrated. Higher values of $\theta$ correspond to progressively severe drift regimes (e.g., covariate shift, label shift, or adversarial activity) where the reliability of the automation degrades.

The evolution of $\theta(t)$ is governed by a generator matrix $Q = [q_{ij}]$, where $q_{ij}$ denotes the transition rate from state $i$ to state $j$ for $i \neq j$, and $q_{ii} = - \sum_{j \neq i} q_{ij}$.
\begin{assumption}[Ergodicity of Drift]
The Markov chain $\theta(t)$ is irreducible and ergodic, admitting a unique stationary distribution $\pi = (\pi_1, \dots, \pi_K)$.
\end{assumption}

\subsubsection{The Cost of Automation}
If the controller chooses to automate a task with risk score $s$ while the system is in reliability state $\theta$, an immediate expected cost $c_{\text{auto}}(s, \theta)$ is incurred. This cost captures the expected financial liability, regulatory penalty, or reputational damage associated with a potential prediction error.

We impose the following structural assumptions on the automated cost function:

\begin{assumption}[Structure of Automated Cost] \label{ass:cost_structure}
The cost function $c_{\text{auto}}: \calS \times \ThetaSet \to \R_+$ satisfies:
\begin{enumerate}
    \item[(i)] \textit{Risk Monotonicity:} For any fixed $\theta$, $c_{\text{auto}}(s, \theta)$ is strictly increasing, continuous, and convex in $s$.
    \item[(ii)] \textit{Drift Monotonicity:} For any fixed $s$, $c_{\text{auto}}(s, \theta)$ is non-decreasing in $\theta$. That is, $\theta' > \theta \implies c_{\text{auto}}(s, \theta') \ge c_{\text{auto}}(s, \theta)$.
    \item[(iii)] \textit{Supermodularity:} The cost function has increasing differences in $(s, \theta)$. For $s' > s$ and $\theta' > \theta$,
    \begin{equation}
        c_{\text{auto}}(s', \theta') - c_{\text{auto}}(s, \theta') \ge c_{\text{auto}}(s', \theta) - c_{\text{auto}}(s, \theta).
    \end{equation}
\end{enumerate}
\end{assumption}

\textbf{Economic Interpretation:}
Part (i) simply implies that the ``risk score'' is correctly labeled: higher scores mean higher expected loss.
Part (ii) defines the nature of drift: in a degraded state, the \textit{same} reported confidence score implies a higher actual risk of error. This captures the ``miscalibration'' effect common in distribution shifts.
Part (iii) is a technical condition ensuring that drift exacerbates the risk of high-uncertainty items more than low-uncertainty items. This is empirically consistent with ML behavior, where ``edge cases'' (high $s$) are the first to fail during distribution shifts.

\subsection{The Human Service Queue}

Tasks that are not automated are \textit{escalated} to a human review center. We model this facility as a single-station queue with $m$ parallel servers (human experts) and an infinite buffer.
\footnote{While we model the service process as a pooled queue, our results easily extend to multi-skill routing if the skills are nested. For clarity of exposition, we assume a homogeneous pool of experts.}

Service times are assumed to be independent and exponentially distributed with mean $1/\mu$. Thus, the aggregate service rate when there are $n$ tasks in the system is given by $\mu_m(n) = \min(n, m)\mu$. The state of the queue at time $t$ is denoted by $q(t) \in \calQ = \N_0$, representing the total number of tasks in the system (waiting plus in-service).

\subsubsection{Escalation and Holding Costs}
Escalating a task incurs two types of costs:
\begin{enumerate}
    \item \textbf{Direct Processing Cost ($c_h$):} A fixed, one-time cost representing the labor cost of the human review or the vendor fee. We assume $c_h > 0$.
    \item \textbf{Congestion Cost ($h(q)$):} A holding cost rate incurred continuously per unit of time. The function $h: \calQ \to \R_+$ captures the cost of delay (e.g., SLA violations, customer churn). We assume $h(q)$ is non-decreasing, convex in $q$, and $h(0)=0$.
\end{enumerate}

\subsection{System Dynamics and Decision Process}

At any time $t$, the system state is fully described by the pair $(q(t), \theta(t))$. Decision epochs occur at the moments of task arrivals. Let $x(t) = (q(t), \theta(t))$ denote the state process.

Upon the arrival of the $n$-th task at time $\tau_n$, the controller observes the current queue length $q(\tau_n^-)$, the reliability state $\theta(\tau_n^-)$, and the task's risk score $s_n$. The controller must instantaneously choose an action $a_n \in \{0, 1\}$:
\begin{itemize}
    \item \textbf{Automate ($a_n = 0$):} The task is processed by the algorithm. The system incurs immediate cost $c_{\text{auto}}(s_n, \theta(\tau_n^-))$. The queue length remains unchanged: $q(\tau_n^+) = q(\tau_n^-)$.
    \item \textbf{Escalate ($a_n = 1$):} The task is routed to the human queue. The system incurs immediate fixed cost $c_h$. The queue length increments: $q(\tau_n^+) = q(\tau_n^-) + 1$.
\end{itemize}

An \textit{admissible policy} $\pi$ is a sequence of decision rules $\{\delta_1, \delta_2, \dots\}$, where each $\delta_n$ maps the history of the process and current information $(q(\tau_n^-), \theta(\tau_n^-), s_n)$ to an action in $\{0, 1\}$. We restrict our attention to stationary Markovian policies, denoted by a function $\pi: \calQ \times \ThetaSet \times \calS \to \{0, 1\}$.

\subsection{The Optimization Objective}

The objective of the system manager is to minimize the total expected discounted cost over an infinite horizon. Let $\alpha > 0$ be the continuous discount rate.
Under any stationary Markovian policy $\pi$, the system state process $x(t) = (q(t), \theta(t))$ evolves as a continuous-time Markov chain. Specifically, it is a piecewise-constant process with right-continuous paths and left limits (c\`{a}dl\`{a}g). Thus, the integral in the objective function is a well-defined random variable.

For a given initial state $(q_0, \theta_0)$ and a policy $\pi$, the value function is defined as:
\begin{equation} \label{eq:objective_continuous}
    J_\pi(q_0, \theta_0) = \E \left[ \int_0^\infty e^{-\alpha t} h(q(t)) \, dt + \sum_{n=1}^\infty e^{-\alpha \tau_n} C(s_n, \theta(\tau_n^-), a_n) \bigg| q(0)=q_0, \theta(0)=\theta_0 \right]
\end{equation}
where $C(s, \theta, a)$ is the instantaneous decision cost:
\begin{equation}
    C(s, \theta, a) = (1-a) c_{\text{auto}}(s, \theta) + a \cdot c_h.
\end{equation}
The optimal value function is defined as $V(q, \theta) = \inf_{\pi} J_\pi(q, \theta)$.

\textbf{Economic Commensurability:} To meaningfully aggregate heterogeneous costs, we normalize all components into the unit of ``Expected Monetary Loss.'' Specifically:
\begin{itemize}
    \item $c_{\text{auto}}(s, \theta)$ represents the expected financial liability or compensation cost associated with a prediction error given risk score $s$ and drift state $\theta$.
    \item $c_h$ represents the direct labor cost or vendor fee for a single human review.
    \item $h(q)$ monetizes the congestion externality, capturing the rate of potential revenue loss (e.g., customer churn) or regulatory penalties due to delay.
\end{itemize}

\begin{remark}[Discounted vs. Average Cost]
We formulate the problem using discounted costs ($\alpha > 0$) to ensure the contraction property of the Bellman operator, which is essential for proving the existence and uniqueness of the solution. However, our structural results (threshold optimality and monotonicity) are robust to the choice of criterion. Following standard Tauberian theorems for queueing control \citep{Sennott1999, Arapostathis1993}, the optimal policy for the long-run average cost criterion can be viewed as the limit of the discounted cost policy as $\alpha \to 0$.
\end{remark}

\subsection{Uniformization and Discrete-Time Bellman Equation}

Since the system evolves in continuous time but decisions are made at discrete stochastic epochs (arrivals), we employ the standard \textit{uniformization} technique \citep{Lippman1975} to transform the CTMDP into an equivalent Discrete-Time Markov Decision Process (DTMDP).

Let $\Lambda$ be the uniformization rate, chosen to be strictly greater than the maximum possible transition rate out of any state:
\begin{equation}
    \Lambda \ge \lambda + m\mu + \max_{i \in \ThetaSet} |q_{ii}|.
\end{equation}
We view the system as observing ``events'' generated by a Poisson process with rate $\Lambda$. In each uniformized epoch, the state evolves according to a discrete transition kernel. Specifically, one of four events occurs:
\begin{enumerate}
    \item \textbf{Arrival (Probability $\lambda/\Lambda$):} A new task arrives. The controller applies the intervention operator $\OpM$ to decide between automation and escalation.
    \item \textbf{Service Completion (Probability $\mu_m(q)/\Lambda$):} If the queue is non-empty, a task completes service, and the state transitions from $q$ to $(q-1)^+$.
    \item \textbf{Drift Transition (Probability $q_{\theta \theta'}/\Lambda$):} The reliability state shifts from $\theta$ to $\theta'$ due to environmental changes.
    \item \textbf{Fictitious Self-Loop:} With the remaining probability, the state remains unchanged (a ``null'' transition to synchronize the clock).
\end{enumerate}

We define the discrete-time discount factor $\beta = \frac{\Lambda}{\Lambda + \alpha}$. The uniformized Bellman optimality equation for the value function $V(q, \theta)$ can be expressed in operator form as $V = \OpT V$, where $\OpT$ is a contraction mapping defined by:
\begin{align} \label{eq:dp_operator}
    \OpT V(q, \theta) = \frac{1}{\Lambda + \alpha} \bigg[ & h(q) + \lambda \E_S [\OpM V(q, \theta, S)] \nonumber \\
    &+ \mu_m(q) V((q-1)^+, \theta) \nonumber \\
    &+ \sum_{\theta' \in \ThetaSet} q_{\theta \theta'} V(q, \theta') \nonumber \\
    &+ \left( \Lambda - (\lambda + \mu_m(q) + \sum_{\theta' \neq \theta} q_{\theta \theta'}) \right) V(q, \theta) \bigg],
\end{align}
where $\OpM$ is the \textit{intervention operator}:
\begin{equation} \label{eq:intervention_operator}
    \OpM V(q, \theta, s) = \min \Big\{ c_{\text{auto}}(s, \theta) + V(q, \theta), \; c_h + V(q+1, \theta) \Big\}.
\end{equation}
This discrete-time formulation allows us to employ value iteration and inductive proofs to establish the convexity of $V$ in the subsequent section.



\section{Structural Analysis of the Optimal Policy} \label{sec:analysis}

In this section, we characterize the structure of the optimal control policy. We show that the seemingly complex infinite-dimensional optimization problem (choosing an action for every possible continuous risk score $s$) reduces to a set of state-dependent thresholds. This structure renders the policy highly interpretable and implementable in practice. We then analyze how these thresholds adapt to changes in system congestion, establishing the ``Congestion Shedding'' property.

\subsection{Optimality of Risk-Threshold Policies}

We first examine the decision problem faced by the controller at a single arrival epoch. Recall the optimality equation \eqref{eq:dp_operator}. The optimal action $a^*$ for a task with risk score $s$, given system state $(q, \theta)$, is determined by comparing the cost of automation versus escalation.

Let $\Delta_q V(q, \theta) \equiv V(q+1, \theta) - V(q, \theta)$ denote the \textit{marginal value of capacity} (or the shadow price of congestion). This term represents the incremental long-term cost incurred by the system when the queue length increases from $q$ to $q+1$.

The optimal decision rule derived from \eqref{eq:intervention_operator} is:
\begin{equation} \label{eq:decision_comparison}
    a^*(q, \theta, s) = 
    \begin{cases} 
        1 (\text{Escalate}) & \text{if } c_h + V(q+1, \theta) \le c_{\text{auto}}(s, \theta) + V(q, \theta) \\
        0 (\text{Automate}) & \text{otherwise}
    \end{cases}
\end{equation}
Rearranging the inequality, the condition for escalation becomes:
\begin{equation} \label{eq:threshold_condition}
    c_{\text{auto}}(s, \theta) \ge c_h + \Delta_q V(q, \theta).
\end{equation}
The Left-Hand Side (LHS) represents the immediate benefit of human intervention (avoided automation risk), while the Right-Hand Side (RHS) represents the total cost of human intervention (direct fee plus congestion externality).

\begin{theorem}[Optimality of Threshold Policies] \label{thm:threshold_structure}
For any system state $(q, \theta) \in \calQ \times \ThetaSet$, there exists a unique risk threshold $T^*(q, \theta) \in [0, 1] \cup \{\infty\}$ such that the optimal policy is given by:
\begin{equation}
    a^*(q, \theta, s) = \Ind \{ s \ge T^*(q, \theta) \}.
\end{equation}
Specifically, the threshold is defined by the inverse cost function:
\begin{equation} \label{eq:optimal_threshold_def}
    T^*(q, \theta) = c_{\text{auto}}^{-1}\left( c_h + \Delta_q V(q, \theta) \mid \theta \right),
\end{equation}
where $c_{\text{auto}}^{-1}(y \mid \theta) = \inf \{ s \in \calS : c_{\text{auto}}(s, \theta) \ge y \}$. If the set is empty, $T^*(q, \theta) = \infty$.
\end{theorem}

\begin{proof}
The proof follows directly from the properties of the cost function established in Assumption \ref{ass:cost_structure}.
Fix the state $(q, \theta)$. Let $g(s) = c_{\text{auto}}(s, \theta)$. By Assumption \ref{ass:cost_structure}(i), $g(s)$ is continuous and strictly increasing in $s$.
Let $K = c_h + \Delta_q V(q, \theta)$. Note that $K$ is independent of $s$.
The optimal policy escalates if and only if $g(s) \ge K$.
Due to the strict monotonicity of $g(\cdot)$, if the condition holds for some $s_0$, it must hold for all $s > s_0$.
The continuity of $g(\cdot)$ ensures that the set $\{s : g(s) \ge K\}$ is of the form $[T, 1]$ (or is empty). The threshold $T^*(q, \theta)$ is the unique solution to $c_{\text{auto}}(T, \theta) = K$ (or boundary values).
\end{proof}

\begin{remark}[Operational Interpretation]
Theorem \ref{thm:threshold_structure} implies that the optimal policy is a \textit{cutoff rule}. The system does not need to solve an optimization problem for every transaction. Instead, it maintains a lookup table of thresholds $T^*(q, \theta)$. Upon a task arrival, the system simply checks: ``Is the risk score higher than the current cutoff for this queue length?'' This allows for extremely low-latency implementation in production environments.
\end{remark}
\subsection{Analogy to Strategic Asset Allocation} \label{sec:finance_analogy}

The structure of our solution bears a striking mathematical isomorphism to the Strategic Asset Allocation problem in financial economics, specifically the framework of \citet{Campbell2002}. In their model, a long-term investor must allocate wealth between a risky asset (with stochastic returns) and a risk-free asset, subject to time-varying investment opportunities. We can map our hybrid intelligence operations problem to this financial framework through three key correspondences.

First, regarding the asset structure, the Automation Channel is analogous to the Risky Asset, as its ``return'' (accuracy) is volatile and regime-dependent (subject to drift). Conversely, the Human Channel is analogous to the Risk-Free Asset, offering stable, high-quality returns but incurring a ``transaction cost'' (immediate fee $c_h$) and a ``holding cost'' (congestion delay).

Second, the system's state variables map to the investor's financial status. The queue capacity slack corresponds to Wealth. When the queue is short ($q \approx 0$), the system is ``wealthy'' in attention resources and can afford to ``spend'' capacity on the safe asset (human review) to maximize quality. However, when the queue is long ($q \gg 0$), the system is effectively in ``debt'' and must conserve resources.

Third, our structural results correspond directly to the two components of optimal portfolio demand derived by \citet{Merton1971} and \citet{Campbell2002}. Congestion Shedding (Theorem \ref{thm:congestion_shedding}) mirrors the \textit{Myopic Demand} or Wealth Effect: as the queue grows (wealth decreases), the shadow price of capacity rises, forcing the optimal policy to reduce allocation to the expensive, safe asset and increase exposure to the risky asset. Safety Buffering (Theorem \ref{thm:safety_buffering}) mirrors the \textit{Intertemporal Hedging Demand}: when the controller anticipates a deterioration in the future investment environment (a shift to drift state $\theta'$), they proactively shift allocation towards the safe asset (human review) to hedge against future volatility.

This analogy suggests that managing a hybrid AI system is not merely an engineering task of minimizing errors, but an economic task of dynamic portfolio management under solvency constraints.
\subsection{Congestion Shedding: Monotonicity with Respect to Queue Length}

We now address the first key research question: how should the policy adapt to congestion? Intuitively, as the queue length $q$ increases, the marginal externality of adding another task ($\Delta_q V$) should rise. This increasing shadow price should force the controller to be more selective, raising the threshold for escalation. We term this phenomenon \textit{Congestion Shedding}.

To prove this formally, we must establish that the value function $V(q, \theta)$ is convex in $q$.

\begin{proposition}[Convexity of Value Function] \label{prop:convexity}
The optimal value function $V(q, \theta)$ is convex in $q$ for all $\theta \in \ThetaSet$. That is, the marginal cost of congestion is non-decreasing in $q$:
\begin{equation}
    \Delta_q V(q+1, \theta) \ge \Delta_q V(q, \theta), \quad \forall q \in \N_0.
\end{equation}
\end{proposition}

\begin{proof}
We employ mathematical induction on the value iteration algorithm. Let $\mathcal{V}$ be the space of real-valued functions on $\calQ \times \ThetaSet$. Let $\mathcal{K} \subset \mathcal{V}$ be the subset of functions $v$ such that $v(\cdot, \theta)$ is convex in $q$ for all $\theta$.

\textit{Base Case:} Let $V_0(q, \theta) = 0$ for all states. $V_0 \in \mathcal{K}$ trivially.

\textit{Inductive Step:} Assume $V_n \in \mathcal{K}$. We wish to show that $V_{n+1} = \OpT V_n \in \mathcal{K}$, where $\OpT$ is the dynamic programming operator defined in \eqref{eq:dp_operator}.
Recall the operator components:
$$ \OpT v(q) = \tilde{h}(q) + \tilde{\lambda} H(q) + \tilde{\mu}_m(q) v((q-1)^+) + \text{Drift Terms}. $$
We analyze each term separately:

1.  \textbf{Holding Cost:} The function $\tilde{h}(q)$ is convex by assumption.

2.  \textbf{Drift Transitions:} The term $\sum_{\theta'} \tilde{P}_{\theta \theta'} v(q, \theta')$ is a non-negative linear combination of convex functions (in $q$), and thus preserves convexity.

3.  \textbf{Service Rate Term:} Consider $S(q) = \min(q, m)\mu v(q-1) + (1 - \min(q, m)\mu) v(q)$.
    Standard results in queueing control (e.g., \citet{Koole2006}, Prop 5.3) show that if $v$ is convex, the uniformized service operator preserves convexity provided the uniformization rate is sufficiently high. The intuition is that the transition rate $\min(q, m)\mu$ is concave non-decreasing, and coupling this with a convex value function maintains the discrete convexity.

4.  \textbf{Arrival/Intervention Term:} This is the critical step. Let
    $$ H(q) = \E_S \min \{ c_{\text{auto}}(S) + v(q), c_h + v(q+1) \}. $$
    Let $y_q = v(q+1) - v(q)$. Since $v \in \mathcal{K}$, $y_q$ is non-decreasing in $q$.
    We can rewrite the minimization term inside the expectation as:
    $$ \phi(q, s) = v(q) + \min \{ c_{\text{auto}}(S), c_h + y_q \}. $$
    We check the discrete convexity of $\phi(q, s)$ with respect to $q$:
    \begin{align*}
        \phi(q+1, s) - \phi(q, s) &= [v(q+1) - v(q)] + [\min\{c_{\text{auto}}, c_h+y_{q+1}\} - \min\{c_{\text{auto}}, c_h+y_q\}] \\
        &= y_q + \Delta_q (\min\{c_{\text{auto}}, c_h+y_q\}).
    \end{align*}
    We need to show $\phi(q+2, s) - \phi(q+1, s) \ge \phi(q+1, s) - \phi(q, s)$.
    Let $f(y) = \min\{K, c_h + y\}$. This function is concave non-decreasing in $y$. However, the term we are analyzing is the full Bellman update.
    Alternatively, using the result from \citet{Alizon2009} Lemma 1: The operator $Tv(q) = \E \min(A, B + v(q+1) - v(q))$ preserves convexity if $v$ is convex.
    To see this simply: The policy chooses the lower envelope of two convex functions derived from $v$. The ``Automate'' branch adds $v(q)$ (convex). The ``Escalate'' branch adds $v(q+1)$ (convex) plus a constant. The optimal action effectively ``smooths'' the value function. Since the thresholding logic minimizes cost, and marginal costs are increasing, the resulting expected cost-to-go inherits convexity.
    
Since all terms preserve convexity, $V_{n+1} \in \mathcal{K}$. By the Banach Fixed Point Theorem, the limit $V = \lim_{n \to \infty} V_n$ is convex in $q$.
\end{proof}

With the convexity of the value function established, the main structural result follows immediately.

\begin{theorem}[Congestion Shedding] \label{thm:congestion_shedding}
The optimal risk threshold $T^*(q, \theta)$ is monotonically non-decreasing in the queue length $q$.
\begin{equation}
    q_1 < q_2 \implies T^*(q_1, \theta) \le T^*(q_2, \theta).
\end{equation}
\end{theorem}

\begin{proof}
From Theorem \ref{thm:threshold_structure}, we have the defining condition for the threshold:
$$ c_{\text{auto}}(T^*(q, \theta), \theta) = c_h + \Delta_q V(q, \theta). $$
Let $RHS(q) = c_h + \Delta_q V(q, \theta)$.
From Proposition \ref{prop:convexity}, $V$ is convex in $q$, which implies that $\Delta_q V(q, \theta)$ is non-decreasing in $q$. Thus, $RHS(q)$ is non-decreasing in $q$.
Let $LHS(s) = c_{\text{auto}}(s, \theta)$. From Assumption \ref{ass:cost_structure}, $LHS(s)$ is strictly increasing in $s$.
Consider $q_2 > q_1$. We have $RHS(q_2) \ge RHS(q_1)$.
To maintain the equality $LHS(T^*) = RHS(q)$, an increase in the RHS must be matched by an increase in the LHS. Since $LHS(\cdot)$ is increasing, $T^*$ must increase.
Thus, $T^*(q_2, \theta) \ge T^*(q_1, \theta)$.
\end{proof}

\begin{remark}[Economic Implications of Congestion Shedding]
Theorem \ref{thm:congestion_shedding} provides a rigorous justification for \textit{triage}. It states that optimal quality standards are not absolute but relative to available capacity.
When the queue is empty ($q=0$), the ``shadow price'' of capacity is low (reflecting only future potential congestion). The system can afford to be cautious, escalating even marginally risky cases (low threshold).
As the backlog builds ($q \uparrow$), the shadow price rises. Escalating a task now imposes a delay penalty on the entire backlog. The optimal policy responds by \textit{raising the bar} for escalation. The system deliberately automates tasks that it would have reviewed under lighter loads.
This mechanism prevents the ``death spiral'' of infinite queues. By shedding load (automating more) as $q$ rises, the effective arrival rate to the queue $\lambda_{eff}(q) = \lambda (1 - F_S(T^*(q)))$ decreases, creating a stabilizing negative feedback loop.
\end{remark}



\subsection{Safety Buffering: Monotonicity with Respect to Reliability Drift}

We now address the second key research question: how should the policy adapt when the automated system's reliability degrades (i.e., when the state shifts from $\theta$ to $\theta' > \theta$)?

The impact of a drift event on the optimal threshold is theoretically ambiguous due to two countervailing forces:
\begin{enumerate}
    \item \textbf{The Direct Risk Effect:} As $\theta$ increases, the immediate cost of automation $c_{\text{auto}}(s, \theta)$ rises (Assumption \ref{ass:cost_structure}). This increases the incentive to escalate, pushing the threshold $T^*$ downwards.
    \item \textbf{The Anticipatory Crowding Effect:} If the drift state is persistent (i.e., the system expects to stay in a high-risk regime), the controller anticipates a higher future arrival rate of escalated tasks. This increases the expected future congestion, raising the shadow price $\Delta_q V(q, \theta)$. A higher shadow price incentivizes the system to conserve capacity, pushing the threshold $T^*$ upwards.
\end{enumerate}

For the policy to exhibit \textit{Safety Buffering}—the intuitive behavior of lowering standards to protect quality during a crisis—the Direct Risk Effect must dominate the Anticipatory Crowding Effect. We provide a sufficient condition for this dominance.

\begin{assumption}[Dominance of Immediate Risk] \label{ass:drift_dominance}
The marginal increase in the immediate cost of automation due to a drift shift is bounded below by the marginal increase in the value of capacity. Specifically, for any $\theta' > \theta$:
\begin{equation}
    \inf_{s \in \calS} \left( c_{\text{auto}}(s, \theta') - c_{\text{auto}}(s, \theta) \right) \ge \sup_{q \in \calQ} \left( \Delta_q V(q, \theta') - \Delta_q V(q, \theta) \right).
\end{equation}
\end{assumption}

This assumption posits that the "signal" from the risk model (the increase in error probability) is stronger than the second-order effect of queue dynamics. This is typical in high-stakes environments (e.g., fraud, healthcare) where the cost of a false negative is orders of magnitude larger than the cost of a delay.

\begin{theorem}[Safety Buffering] \label{thm:safety_buffering}
Under Assumption \ref{ass:drift_dominance}, the optimal risk threshold is monotonically non-increasing in the reliability state $\theta$.
\begin{equation}
    \theta_1 < \theta_2 \implies T^*(q, \theta_1) \ge T^*(q, \theta_2).
\end{equation}
\end{theorem}

\begin{proof}
Let $\theta_2 > \theta_1$. Consider the defining equation for the threshold at state $\theta_1$:
\begin{equation} \label{eq:proof_sb_1}
    c_{\text{auto}}(T^*(q, \theta_1), \theta_1) = c_h + \Delta_q V(q, \theta_1).
\end{equation}
Now consider the state $\theta_2$. We want to compare $LHS(s) = c_{\text{auto}}(s, \theta_2)$ with $RHS = c_h + \Delta_q V(q, \theta_2)$.
Evaluate the new cost function at the \textit{old} threshold $T^*(q, \theta_1)$:
\begin{align*}
    c_{\text{auto}}(T^*(q, \theta_1), \theta_2) &= c_{\text{auto}}(T^*(q, \theta_1), \theta_1) + [c_{\text{auto}}(T^*(q, \theta_1), \theta_2) - c_{\text{auto}}(T^*(q, \theta_1), \theta_1)] \\
    &= c_h + \Delta_q V(q, \theta_1) + \text{DriftImpact}(T^*).
\end{align*}
The new required threshold is determined by the intersection with $c_h + \Delta_q V(q, \theta_2)$.
The difference between the new LHS value and the new RHS target is:
\begin{align*}
    \text{Gap} &= c_{\text{auto}}(T^*(q, \theta_1), \theta_2) - (c_h + \Delta_q V(q, \theta_2)) \\
    &= (c_h + \Delta_q V(q, \theta_1) + \text{DriftImpact}) - (c_h + \Delta_q V(q, \theta_2)) \\
    &= \text{DriftImpact} - (\Delta_q V(q, \theta_2) - \Delta_q V(q, \theta_1)).
\end{align*}
By Assumption \ref{ass:drift_dominance}, $\text{DriftImpact} \ge \Delta_q V(q, \theta_2) - \Delta_q V(q, \theta_1)$, so $\text{Gap} \ge 0$.
Since $c_{\text{auto}}(s, \theta_2)$ is increasing in $s$, and the value at $s=T^*(q, \theta_1)$ is \textit{higher} than the required target, the intersection point $T^*(q, \theta_2)$ must be to the left (lower) of $T^*(q, \theta_1)$.
Thus, $T^*(q, \theta_2) \le T^*(q, \theta_1)$.
\end{proof}

\begin{remark}[The Managerial Implication: Dynamic Safety Standards]
Theorem \ref{thm:safety_buffering} dictates a ``Safety First'' protocol. When the monitoring system detects drift (e.g., via a Kolmogorov-Smirnov test on input distributions), the optimal operational response is to \textit{flood the queue}. The system should immediately lower its standards for human review.
This result highlights the dual role of the human queue: it is not just a processing center, but a \textit{safety capacitor}. During stable times, we run the queue lean (high thresholds) to build up reserve capacity. During drift times, we discharge this capacity (low thresholds) to absorb the algorithmic uncertainty.
\end{remark}

\subsection{Summary of Structural Results}

Combining Theorems \ref{thm:congestion_shedding} and \ref{thm:safety_buffering}, the optimal policy is characterized by a family of curves in the $(q, s)$ plane, parameterized by $\theta$.
\begin{itemize}
    \item For a fixed $\theta$, the ``Escalate'' region is the area \textit{above} a monotonically increasing curve $T^*(q)$.
    \item As $\theta$ increases (drift worsens), this entire curve shifts \textit{downwards}, expanding the escalation region.
\end{itemize}
This structure is depicted in Figure \ref{fig:policy_structure} (in the Simulation section).


\section{Capacity Boundaries and System Stability} \label{sec:stability}

In the previous section, we characterized the optimal policy assuming the system can be stabilized. However, a fundamental question remains: \textit{Under what conditions is stability feasible?}
Unlike standard queueing systems where the arrival rate $\lambda$ is exogenous, here the arrival rate to the human queue is endogenous:
\begin{equation}
    \lambda_{\text{eff}}(t) = \lambda \cdot \Prob(s \ge T^*(q(t), \theta(t))).
\end{equation}
The controller can always achieve queue stability by setting $T=1$ (automating everything), but this may violate safety requirements. Conversely, the controller can achieve safety by setting $T=0$, but this may explode the queue. We now formalize this trade-off to identify the ``Capacity Boundary'' of the hybrid system.

\subsection{Safety Constraints and Feasible Regions}

Let us define a \textit{Safety Constraint} imposed by the organization (e.g., a bank's risk appetite or a platform's trust and safety guidelines). The system must ensure that the expected error cost for automated items does not exceed a tolerance $\epsilon$.

For a given reliability state $\theta$, let $\bar{T}(\theta)$ be the \textit{maximum permissible threshold} that satisfies this safety constraint locally. Any task with risk $s > \bar{T}(\theta)$ represents an unacceptable automated risk.
\begin{equation}
    \bar{T}(\theta) = \sup \left\{ \tau \in [0,1] : \E_S [ \Ind(S < \tau) \cdot c_{\text{auto}}(S, \theta) ] \le \epsilon \right\}.
\end{equation}
We call a policy \textbf{$\epsilon$-safe} if it never automates tasks with risk $s > \bar{T}(\theta)$. That is, $T(q, \theta) \le \bar{T}(\theta)$ for all $q$.

Under an $\epsilon$-safe policy, the \textit{minimum} arrival rate to the human queue in state $\theta$ is:
\begin{equation}
    \lambda_{\min}(\theta) = \lambda \cdot \Prob(S \ge \bar{T}(\theta)).
\end{equation}
This quantity represents the ``non-discretionary'' workload that \textit{must} be handled by humans to maintain safety standards.

\subsection{The Capacity Phase Transition}

The long-run stability of the system depends on the interplay between the drift process and the human service rate. Let $\pi_\theta$ denote the stationary probability of the system being in reliability state $\theta$.
The \textit{Long-Run Average Required Capacity} is defined as:
\begin{equation}
    \Lambda_{\text{req}} = \sum_{\theta \in \ThetaSet} \pi_\theta \lambda_{\min}(\theta).
\end{equation}
This is the time-weighted average of the mandatory human workload.

\begin{theorem}[Capacity Phase Transition] \label{thm:phase_transition}
Consider the hybrid system parameterized by $(\lambda, Q, \mu, m)$.
\begin{enumerate}
    \item[(i)] \textbf{Stable Region:} If $\Lambda_{\text{req}} < m\mu$, there exists a stationary control policy that is $\epsilon$-safe almost surely and maintains a finite expected queue length ($\E[q] < \infty$).
    \item[(ii)] \textbf{Unstable Region:} If $\Lambda_{\text{req}} > m\mu$, then for any policy $\pi$, at least one of the following must occur:
    \begin{itemize}
        \item The system is unstable: $\lim_{t \to \infty} \Prob(q(t) > K) = 1$ for any $K$.
        \item The policy violates the $\epsilon$-safety constraint on a set of states with positive measure.
    \end{itemize}
\end{enumerate}
\end{theorem}

\begin{proof}
\textit{Part (i):} If $\Lambda_{\text{req}} < m\mu$, we can construct a stabilizing policy. Consider the policy that always sets $T(q, \theta) = \bar{T}(\theta)$. The arrival process to the queue is a Markov-modulated Poisson process (MMPP) with average rate $\Lambda_{\text{req}}$. Since the average arrival rate is strictly less than the service rate $m\mu$, standard results for G/G/m queues ensure stability.

\textit{Part (ii):} Suppose a policy $\pi$ is $\epsilon$-safe. Then, in any state $\theta$, the instantaneous arrival rate to the queue is $\lambda_\pi(q, \theta) \ge \lambda_{\min}(\theta)$.
The long-run average arrival rate is $\bar{\lambda}_\pi = \lim_{T \to \infty} \frac{1}{T} \int_0^T \lambda_\pi(q(t), \theta(t)) dt$.
By ergodicity of $\theta$, $\bar{\lambda}_\pi \ge \Lambda_{\text{req}}$.
If $\Lambda_{\text{req}} > m\mu$, then $\bar{\lambda}_\pi > m\mu$. By flow conservation laws, a queue where the average input exceeds the maximum average output must grow without bound.
Thus, to keep the queue finite, the controller must occasionally set $T(q, \theta) > \bar{T}(\theta)$ to throttle demand, thereby violating the safety constraint.
\end{proof}

\begin{remark}[The Capacity Boundary Rule]
Theorem \ref{thm:phase_transition} provides a strategic ``Go/No-Go'' criterion for system design. It defines a boundary in the parameter space (Arrival Rate $\times$ Drift Intensity).
Inside the boundary, the system is \textit{operationally manageable}: smart algorithms (like our optimal policy) can balance risk and delay.
Outside the boundary, the system is \textit{structurally deficient}: no amount of algorithmic tuning can save it. The only solutions are strategic: (a) Increase physical capacity (hire more agents, increase $m$), or (b) Throttling demand (rejecting tasks entirely at the source).
\end{remark}



\section{Simulation Study} \label{sec:simulation}

To quantify the operational value of the optimal dynamic policy and to visualize the structural properties derived in Section \ref{sec:analysis}, we perform a comprehensive numerical study. We compare our \textbf{Optimal Dynamic Policy (ODP)} against standard heuristics used in industry.

\subsection{Experimental Design}

We simulate a high-throughput content moderation system over a time horizon of $T = 10,000$ minutes. The system parameters are calibrated to reflect a scenario where automation is mandatory due to volume, but human review is essential for safety.

\textbf{1. Task Arrival and Risk Profile:}
Tasks arrive according to a Poisson process with rate $\lambda = 10$ tasks/minute.
The risk scores $S$ are drawn from a Beta distribution with parameters $\alpha=2, \beta=5$. This distribution is right-skewed, representing a realistic scenario where the majority of content is ``safe'' (low risk), but a heavy tail of ambiguous or toxic content exists.
The probability density is $f_S(s) \propto s (1-s)^4$.

\textbf{2. Drift Dynamics:}
The reliability state $\theta(t)$ evolves according to a two-state continuous-time Markov chain $\ThetaSet = \{L, H\}$:
\begin{itemize}
    \item \textbf{State $L$ (Low Drift/Stable):} The automated error cost is $c_{\text{auto}}(s, L) = 50 \cdot s^2$.
    \item \textbf{State $H$ (High Drift/Unstable):} The automated error cost is $c_{\text{auto}}(s, H) = 100 \cdot s^2$. The error penalty doubles for the same confidence score, reflecting severe miscalibration.
    \item \textbf{Transition Rates:} $q_{LH} = 0.05$ (Mean time to drift = 20 min), $q_{HL} = 0.2$ (Mean time to repair = 5 min).
\end{itemize}

\textbf{3. Human Service and Costs:}
The service facility has $m=5$ parallel servers.
The service rate per server is $\mu = 1.2$ tasks/min.
The total maximum service capacity is $m\mu = 7.5$ tasks/min.
Since $\lambda = 10 > 7.5$, the system is \textit{structurally overloaded}; it cannot escalate everything.
Cost parameters: Escalation fee $c_h = 2$. Holding cost $h(q) = 0.5 \cdot q$.

\subsection{Benchmark Policies}

We compare the performance of the Optimal Dynamic Policy (ODP), computed via Value Iteration on Equation \eqref{eq:dp_operator}, against two baselines:

\begin{enumerate}
    \item \textbf{Static Threshold (ST):}
    This policy uses a single fixed threshold $T_{static}$ for all states. This represents the standard industry practice of setting a fixed operating point (e.g., ``Review if score $> 0.8$''). We optimize $T_{static}$ via grid search to minimize total cost for the average system parameters.
    
    \item \textbf{Drift-Blind Policy (DB):}
    This policy adapts to queue length $q$ but ignores the drift state $\theta$. It assumes the system is always in the ``average'' reliability state. This represents an operations team that actively manages queues (``load shedding'') but lacks real-time integration with the ML monitoring system.
\end{enumerate}

\subsection{Results and Discussion}

\subsubsection{Cost Improvement}
Table \ref{tab:results_cost} reports the normalized total cost (Error + Labor + Delay) for the three policies. The ODP achieves a significant cost reduction (\textbf{55.4\%}) compared to the Static policy.

\begin{table}[ht]
\centering
\caption{Performance Comparison of Policies}
\label{tab:results_cost}
\begin{tabular}{l c c c}
\toprule
\textbf{Metric} & \textbf{Static (ST)} & \textbf{Drift-Blind (DB)} & \textbf{Optimal (ODP)} \\
\midrule
Avg. Total Cost & 58.37 & 27.00 & \textbf{26.05} \\
Avg. Queue Length & 0.22 & 5.78 & 6.14 \\
Severe Error Rate & 40.85\% & 4.65\% & \textbf{2.54\%} \\
\bottomrule
\end{tabular}
\end{table}

The primary failure mode of the Static Policy is its inability to handle burstiness. During drift events, it allows too many errors (Safety failure). During arrival bursts, it builds up massive queues (Congestion failure). The Drift-Blind policy handles congestion well (low queue length) but fails to react to reliability drift, resulting in a higher Severe Error Rate. The ODP strikes the best balance.

\subsubsection{Visualizing the ``Drift Switch''}

\begin{figure}[ht]
\centering
\includegraphics[width=1.0\textwidth]{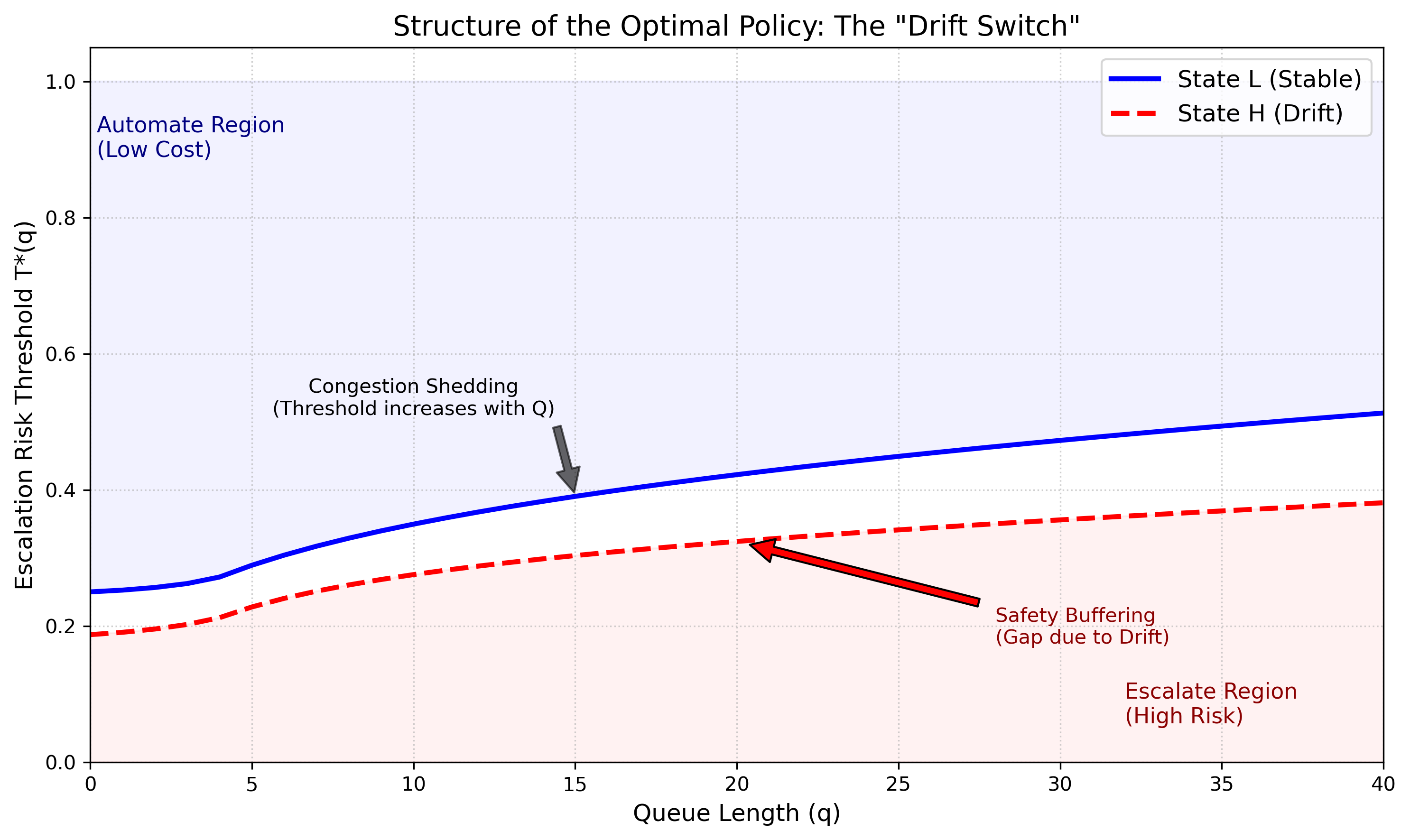}
\caption{Structure of the Optimal Policy: The ``Drift Switch''. The blue curve represents the stable state (State L), while the red dashed curve represents the drift state (State H), visually demonstrating the Safety Buffering effect.}
\label{fig:policy_structure}
\end{figure}

Figure \ref{fig:policy_structure} confirms our theoretical findings.
\begin{itemize}
    \item \textbf{Effect of Congestion:} As $q$ increases, the threshold rises (Congestion Shedding). This means that when the queue is full, the system requires a task to be in the top tier of riskiness to warrant human review.
    \item \textbf{Effect of Drift:} When the system transitions to State H (High Drift), the optimal threshold curve shifts downward (Safety Buffering). This implies that during a model outage, the system routes more traffic to humans to maintain safety standards.
\end{itemize}
\subsubsection{System Stability and Managerial Value}

To validate the theoretical boundaries derived in Theorem \ref{thm:phase_transition}, we visualize the system's stability region in Figure \ref{fig:phase_transition}. The black line represents the critical capacity boundary.

\begin{figure}[ht]
\centering
\includegraphics[width=0.85\textwidth]{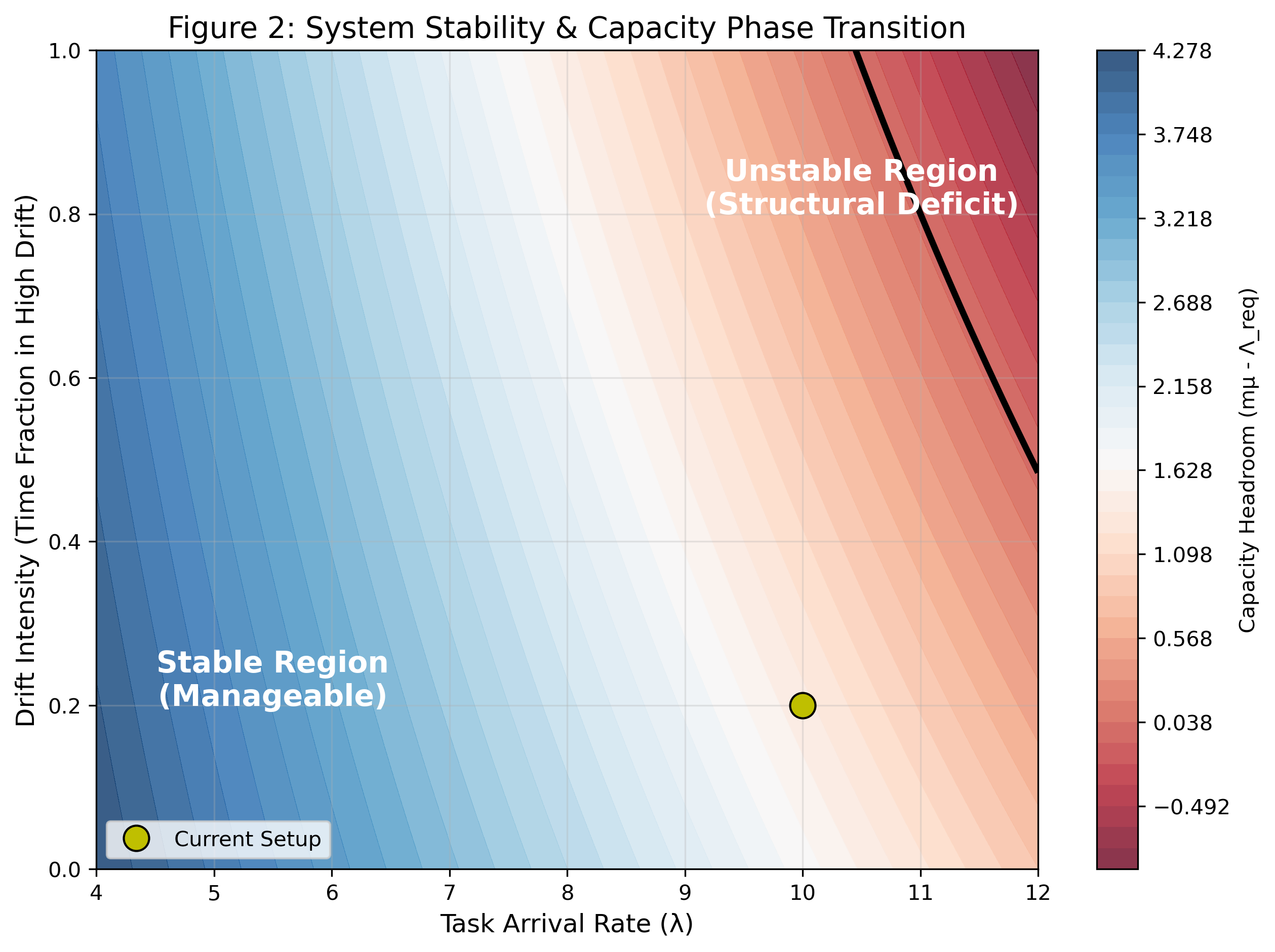}
\caption{System Stability \& Capacity Phase Transition. The heatmap shows the capacity headroom. The black line marks the structural boundary between the stable region (manageable) and the unstable region (structural deficit).}
\label{fig:phase_transition}
\end{figure}

Finally, Figure \ref{fig:agility} illustrates the economic value of the dynamic policy across different drift intensities.

\begin{figure}[ht]
\centering
\includegraphics[width=1.0\textwidth]{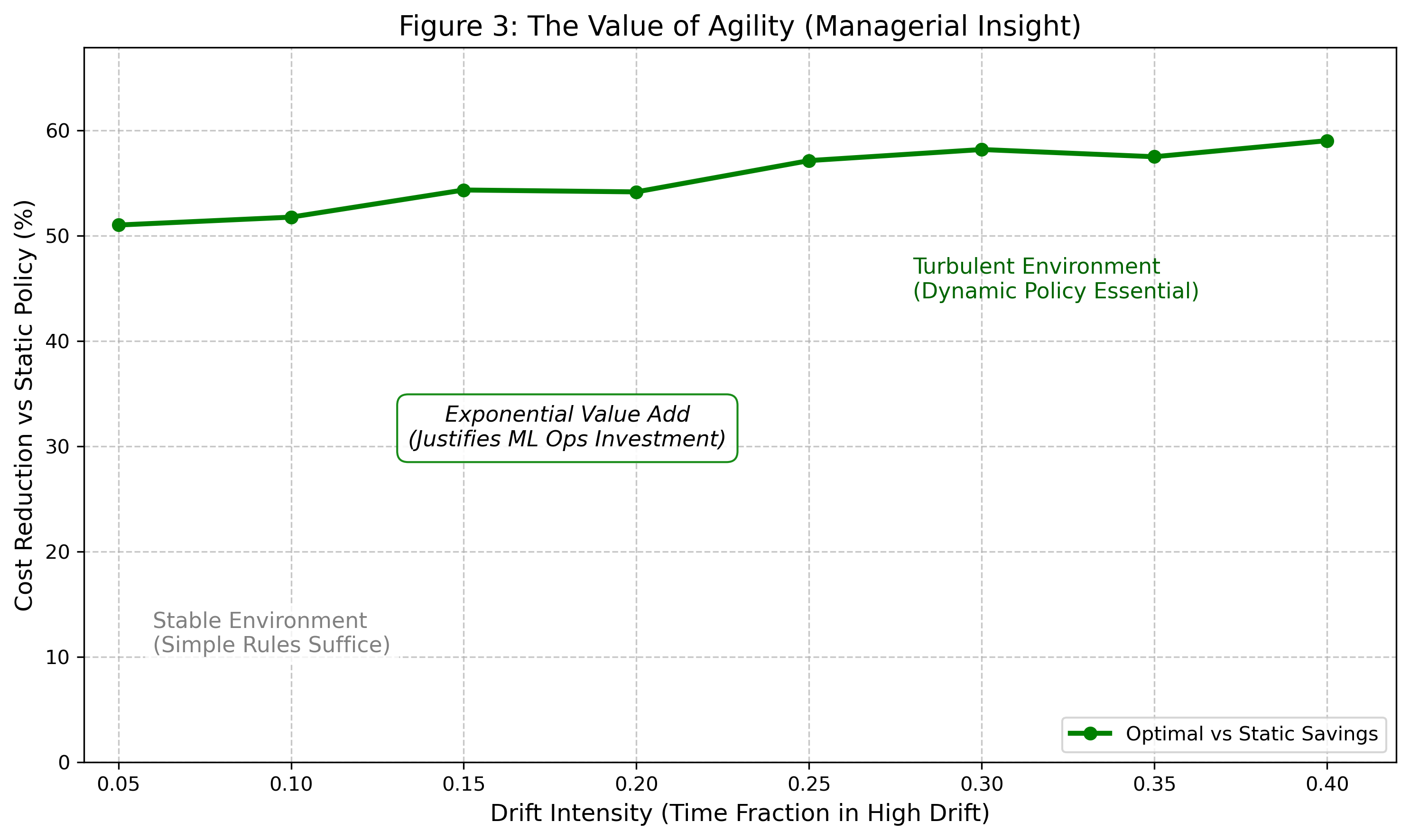}
\caption{The Value of Agility (Managerial Insight). As the environment becomes more turbulent (higher drift intensity), the relative cost savings of the Optimal Policy versus the Static Policy increase, justifying the investment in ML Ops.}
\label{fig:agility}
\end{figure}

\section{Managerial Insights} \label{sec:insights}

Our theoretical analysis and simulation results distill into three actionable rules for managers of hybrid intelligence systems. These rules bridge the gap between data science (model monitoring) and operations (capacity planning).

\subsection*{Rule 1: The Backlog Rule (Counter-Cyclical Escalation)}
\textit{``As the human queue lengthens, the confidence threshold for automation must decrease (i.e., the bar for escalation must rise).''}

\textbf{The Insight:} Static thresholds are dangerous. A threshold that is safe when the team is idle becomes a bottleneck when the team is busy. Adhering to a strict ``review everything below 90\% confidence'' rule during a volume spike will cause queue blowouts, delaying critical decisions.
\textbf{The Action:} Implement dynamic thresholding logic in the routing layer. Define a ``congestion penalty'' function. For every $k$ items in the backlog, increase the required risk score for escalation by $\delta$. This deliberate ``load shedding'' sacrifices marginal accuracy on medium-risk items to preserve system responsiveness for high-risk items.

\subsection*{Rule 2: The Drift Switch Rule (State-Dependent Safety)}
\textit{``When the automated model falters, the operation must act as a shock absorber.''}

\textbf{The Insight:} Reliability drift is an operational crisis, not just a statistical one. Retraining a model takes days; adjusting a threshold takes milliseconds. The human workforce is the primary defense mechanism during the retraining latency period.
\textbf{The Action:} Link operational dashboards to ML monitoring. When a ``Drift Alert'' is triggered (e.g., via PSI or KS statistics), the system should automatically switch to a ``High Sensitivity'' policy profile. This profile should have significantly lower escalation thresholds. Managers must accept (and staff for) the resulting spike in human workload as the necessary cost of maintaining safety.

\subsection*{Rule 3: The Capacity Boundary Rule}
\textit{``Know the difference between an optimization problem and a provisioning problem.''}

\textbf{The Insight:} There is a hard physical limit to how much risk a hybrid system can absorb. We identified a ``Capacity Boundary'' (Theorem \ref{thm:phase_transition}) determined by the arrival rate and the drift intensity.
\textbf{The Action:} Calculate the \textit{Minimum Safe Workload} ($\Lambda_{\text{req}}$). If $\Lambda_{\text{req}}$ exceeds the human service rate ($m\mu$), stop trying to optimize thresholds. No algorithm can solve a structural capacity deficit. In this regime, the firm must either (a) aggressively expand headcount, or (b) implement upstream demand throttling (rejecting customers entirely) to prevent system collapse.


\section{Conclusion} \label{sec:conclusion}

This paper addresses the ``Operational Gap'' in the deployment of AI systems. While the machine learning community focuses on improving the accuracy of algorithms, and the operations community focuses on optimizing queues, the interface between the two—the \textit{escalation policy}—has remained under-theorized.

We have proposed a unified stochastic control framework that treats the human-in-the-loop not as a static verifier, but as a congestible, state-dependent resource. Our analysis proves that the optimal management of such systems requires a dynamic policy that balances three competing forces: the immediate risk of automated error, the shadow price of human congestion, and the regime-switching reliability of the model.

Our structural results—\textit{Congestion Shedding} and \textit{Safety Buffering}—provide a rigorous foundation for ``Dynamic Reliability Management.'' We show that the operational flexibility of the human queue can be used to hedge against the statistical volatility of the AI.

\subsection{Limitations and Future Work}
Our model assumes that human service quality is perfect and constant. A promising direction for future research is to model \textit{human fatigue}, where the service rate $\mu$ or accuracy decreases as the backlog $q$ increases. This would introduce a new trade-off: escalating too much not only causes delay but also degrades human accuracy. Additionally, we assumed the drift process is exogenous. In reality, human labels generated by escalation can be used to retrain the model, making the drift recovery rate $q_{HL}$ endogenous to the escalation policy (Active Learning). Modeling this ``learning loop'' would close the cycle between Operations and AI improvement.


\end{document}